\documentclass[a4paper,12pt]{article}

\usepackage{caption}

\usepackage[a4paper,width=160mm,top=25mm,bottom=30mm]{geometry}
\usepackage[utf8]{inputenc}
\usepackage[T1]{fontenc}
\usepackage[english]{babel}

\usepackage{amsthm}

\theoremstyle{plain}
\newtheorem{theorem}{Theorem}
\newtheorem{proposition}{Proposition}[section]
\newtheorem{corollary}[proposition]{Corollary}
\newtheorem{lemma}[proposition]{Lemma}

\usepackage{hyperref}

\usepackage{enumerate}
\usepackage[shortlabels]{enumitem}
\setlist[itemize,enumerate]{itemsep=4pt,topsep=4pt,parsep=0pt}

\theoremstyle{remark}

\theoremstyle{definition}

\usepackage{amsmath}
\usepackage{amssymb}
\usepackage{dsfont} 
\usepackage{mathtools}

\newcommand{\eps}{\varepsilon}

\newcommand{\R}{\mathds{R}}

\newcommand{\PP}{\mathbf{P}}
\renewcommand{\P}[1]{\PP \left [ #1 \right ]}
\newcommand{\Pp}[1]{\PP_p \left [ #1 \right ]}

\renewcommand{\cal}{\mathcal}
\renewcommand{\emph}{\textsf}

\usepackage{tabularx}
\usepackage{graphicx}

\usepackage{enumerate}
\mathtoolsset{showonlyrefs}

\usepackage{constants}
\newconstantfamily{small}{
symbol=c}

\title{Crossing probabilities for Voronoi percolation} \author{Vincent
  \textsc{Tassion}}
\date{\today} 

\begin{document}
\maketitle

\begin{abstract}
  We prove that the standard Russo-Seymour-Welsh theory is valid for
  Voronoi percolation. This implies that at criticality the crossing
  probabilities for rectangles are bounded by constants depending only
  on their aspect ratio. This result has many consequences, such as
  the polynomial decay of the one-arm event at criticality.
\end{abstract}

\section*{Introduction}

Russo-Seymour-Welsh (RSW) theory is one of the most important tools in the study
of planar percolation. A RSW-result generally refers to an inequality that
provides a bound on the probability to cross rectangles in the long direction,
assuming a bound on the probability to cross squares (or rectangles in the short
direction). Heuristically, this inequality is obtained by ``gluing'' together
square-crossings in order to obtain a crossing in a long rectangle.

Such results were first obtained for Bernoulli percolation on a lattice with a
symmetry assumption \cite{russo1978note}, \cite{seymour1978percolation},
\cite{russo1981critical}, \cite{kesten1982percolation}. For continuum
percolation in the plane, a RSW-result has been proved in~\cite{roy1990russo}
for open crossing events, and in~\cite{alexander1996rsw} for closed crossing
events. A RSW-theory has been recently developed for FK-percolation, see e.g.\@
\cite{beffara2012self, duminil2011connection, duminil2014continuity}. For
Voronoi percolation and for Bernoulli percolation on a lattice without symmetry,
weaker versions of the standard RSW-result have been proved
in~\cite{bollobas2006critical} and ~\cite{bollobas2010percolation},
respectively. Some RSW-techniques have also been recently developed for
Bernoulli percolation on quasi-planar graphs, called slabs. The case of a thin
slab is treated in~\cite{DNS12}. The study of thick slabs in
\cite{duminil2013absence} involves methods similar to those in the present
paper.

At criticality, RSW-results imply the following statement, called the
\emph{box-crossing property}: the crossing probability for any rectangle remains
bounded between $c$ and $1-c$, where $c>0$ is a constant depending only on the
aspect ratio of the rectangle (in particular it is independent of the scale).
For the terminology, we follow~\cite{grimmett2013} where the box-crossing
property is established for Bernoulli percolation on isoradial graphs.

For Bernoulli percolation, the original proof of the Russo-Seymour-Welsh theorem
relies on the spatial Markov property and independence: assuming that a
left-right crossing exists in a square, one can first find the lowest one by
exploring the region below it. Then, the configuration can be sampled
independently in the unexplored region (above the path). This argument does not
apply directly to models with spatial dependence. Voronoi percolation is one of
the most famous model in which the RSW theory and its consequences were expected
but the standard proof did not apply. A major breakthrough was achieved by
Bollobás and Riordan~\cite{bollobas2006critical}, who develop a clever
renormalization method and proved a weak form of RSW (see below). This was
strong enough for their purpose (to show that the critical probability is 1/2),
but too weak to imply all the consequences of the standard RSW. In the present
paper, we prove a stronger RSW for Voronoi. Our proof has some general features
in common with that for the weaker form of Bollobás and Riordan, but also has
major differences. The general structure is similar, since we also use
renormalization, but our scheme does not involve the same quantities as in the
Bollobás-Riordan approach. As in \cite{bollobas2006critical}, our proof does not
rely on exploration or specific properties of Voronoi percolation; it extends to
a large class of percolation models. We present our argument in the framework of
Voronoi percolation, since it is the archetypal example for which the ``lowest
path'' argument does not apply, due to local dependencies.


First introduced in the context of first passage percolation
\cite{vahidi1993upper}, planar Voronoi percolation has been an active area of
research, see for example \cite{bollobas2006critical, benjamini1998conformal,
  aizenman1998scaling, balister2005percolation}. It can be defined by the
following two-step procedure. (A more detailed definition will be given in
Section~\ref{crossing:sec:voronoi-percolation}.) First, construct the Voronoi
tiling associated to a Poisson point process in $\R^2$ with intensity $1$. Then,
color independently each tile black with probability $p$ and white with
probability $1-p$. The self-duality of the model for $p=1/2$ suggests that the
critical value is $p_c=1/2$. The first proof of this, given by
\cite{bollobas2006critical}, required some RSW-like bounds. Instead of a
standard formulation, that paper gave the following, weaker version of the
theorem: for $\rho\ge1$ and $s\ge1$, let $f_s(\rho)$ be the probability that
there exists a left-right black crossing in the rectangle $[0,\rho
  s]\times[0,s]$. For fixed $0<p<1$, they proved that $\inf_{s>0}f_s(1)>0$
implies that $\limsup_{s\to \infty}f_s(\rho)>0$ for all $\rho\ge1$. In other
words, a RSW-result has been obtained for arbitrarily large scale, but not for
all scales. This result was strengthened in~\cite{van2008box}: they proved that
the condition $\limsup_{s\to \infty} f_s(\rho)$ for some $\rho>0$ suffices to
imply that $\limsup_{s\to \infty}f_s(\rho)>0$ for every $\rho>0$. Our main
result is the following standard RSW for Voronoi percolation.
\begin{theorem}
  \label{crossing:thm:main}
  Let $0<p<1$ be fixed. If $\inf_{s\ge1} f_s(1)>0$, then  we have
  for
  all $\rho\ge 1$ $\inf_{s\ge1}  f_s(\rho)>0$.
\end{theorem}

Our work also proves the ``high-probability''-version of RSW, stated
in Theorem~\ref{crossing:thm:highProb}. As we will see in
Section~\ref{sec:theor-impl-theor}, this second result can be derived
from Theorem~\ref{crossing:thm:main}.
\begin{theorem}
    \label{crossing:thm:highProb}
  Let $0<p<1$ be fixed. If $\lim_{s\to\infty} f_s(1)=1$, then we have for
  all $\rho\ge 1$ $\lim_{s\to\infty} f_s(\rho)=1$.
\end{theorem}

At criticality (when $p=1/2$), it is known that $f_s(1)=1/2$ for all
$s$, and Theorem~\ref{crossing:thm:main} above implies the following new results.

\begin{theorem}\label{thm:2}
  Consider Voronoi percolation at $p=1/2$. Then the following holds.
  \begin{description}
  \item[{1. [Box crossing property]}] For all $\rho>0$, there exists
    $c(\rho)>0$ such that
    \begin{equation}
       c(\rho)<f_s(\rho)<1-c(\rho),\quad \text{for all $s\ge 1$}. 
    \end{equation}

  \item[{2. [Polynomial decay of the 1-arm event]}] Let $\pi_1(s,t)$ be
    the probability that there exists a black path from $[-s,s]^2$ to
    the boundary of $[-t,t]^2$. There exists
    $\eta>0$, such that, for every $1\le  s <t$,
    \begin{equation}
      \pi_1(s,t)\le \left(\frac s t \right )^\eta.      
    \end{equation}
   \end{description}
\end{theorem}

  \begin{description}
  \item[Remark 1.] Theorem~\ref{thm:2} is merely one potential
    application of Theorem~\ref{crossing:thm:main}. In the case of
    Bernoulli percolation, RSW bounds have many consequences. These
    include Kesten’s scaling relations, the computation of the
    universal exponents, and tightness of the interfaces in the study
    of scaling limits, to name a few. We expect that similar results
    can be derived for Voronoi percolation using Theorem~\ref{crossing:thm:main}.
  \item[Remark 2.] Our proof is not restricted to Voronoi percolation, and
    Theorem~\ref{crossing:thm:main} extends to a large class of planar
    percolation models. In order to help the reader interested in
    applying the technique of the present paper in a different
    context, we isolate in the framework of Voronoi percolation the
    three sufficient properties that we use (see
    Section~\ref{crossing:sec:voronoi-percolation} for the main
    definitions and notation):
    \begin{description}
    \item[\it(i)  Positive association.] If $\cal A,\, \cal B$ are two
      (black-)increasing events, we have $\P{\cal A\cap \cal B}\ge \P{\cal
        A}\P{\cal B}$.
    \item[\it(ii) Invariance properties.] The measure is invariant under
      translation, $\pi/2$-rotation and horizontal reflection.
    \item[\it(iii) Quasi-independence.] We have
      \begin{equation}
        \label{crossing:eq:4}
        \lim_{s\to \infty} \sup_{\substack{\cal A \in \sigma(A_{2s,4s})\\
            \cal B\in\sigma(\R^2\setminus A_{s,5s})}} |\P{\cal A \cap \cal B}-  \P{\cal A}\P{\cal B}|=0,
      \end{equation}
      where 
      $\sigma(S)$ denotes the sigma-algebra defined by the events
      measurable with respect to the coloring in $S$,
      \mbox{$S\subset\R^2$}.
    \end{description}

  \item[Remark 3.] The proof of the weak RSW of Bollobás and Riordan
    \cite{bollobas2006critical} also applies to a large class of
    model, and require only properties similar to (i), (ii) and (iii).
    With our approach, we also obtain a simple proof of the weak RSW
    of Bollobás and Riordan, using only Properties (i) and (ii); this
    proof is given in the comment at the end of
    Section~\ref{crossing:sec:glueing}. Interestingly, our proof of
    the weak RSW does not use any independence property. This suggests
    also that the standard RSW of Theorem~\ref{crossing:thm:main}
    could be proved using only positive association and invariance
    under some symmetries.
\end{description}

\section{Voronoi percolation}
\label{crossing:sec:voronoi-percolation}

\subsection{Definitions and  notation}

\paragraph{General notation.}
The Lebesgue measure of a measurable set $A\subset\R^2$ is denoted by
$\mathrm{vol}(A)$. The cardinality for a set $S$ is denoted by $|S|$
(with $|S|=+\infty$ if $S$ is infinite). We write $\mathrm d(u,v)$ the
Euclidean distance between two points $u,v\in\R^2$. Finally, for $0\le
s \le t <\infty$, we set \[B_s=[-s,s]^2\quad\text{and} \quad A_{s,t}=B_t\setminus B_s.\]

\paragraph{Voronoi tilings.} Let $\Omega$ be the set of all subsets
$\omega$ of $\R^2$ such that the intersection of $\omega$ with any
bounded set is finite. Equip $\Omega$ with the sigma-algebra generated
by the functions $\omega \mapsto |\omega\cap A|$, $A\subset \R^2$. To
each $\omega\in \Omega$ corresponds a Voronoi tiling, defined as
follows. For every $z \in \omega$, let $V_z$ be the \emph{Voronoi
  cell} of $z$, defined as the set of all points $v\in \R^2$ such that
$\mathrm{d}(v,z) \le \mathrm{d}(v,z')$ for all $z'\in \omega$. The
family $(V_z)_{z\in\omega}$ of all the cells forms a tiling of the
plane.

\paragraph{Voronoi percolation.}
Given a parameter $p\in [0,1]$, define the Voronoi percolation process
as follows. Let $X$ be a Poisson point process in $\R^2$ with density
$1$; for completeness, we recall that $X$ is defined as a random
variable in $\Omega$ characterized by the following two properties.
For every measurable set $A$ (with finite measure), $X\cap A$ contains
exactly $k$ points with probability
\begin{equation}
  \frac{\mathrm{vol}(A)^k}{k!} \exp(-\mathrm{vol}(A)),  
\end{equation}
and the random variables $|X\cap A_1|,\ldots,|X\cap A_n|$ are independent
whenever $A_1,\ldots,A_n$ are disjoint measurable sets. Declare each point $z\in
X$ to be black with probability $p$, and white with probability $1-p$,
independently of each other and of the variable $X$. Define then $X_{\mathrm b}$
and $X_{\mathrm w}$ to be respectively the set of black and white points in $X$.
Notice that we could have equivalently defined $X_{\mathrm b}$ and $X_{\mathrm
  w}$ as two independent Poisson processes with density $p$ and $1-p$, and then
formed $X=X_{\mathrm b}\cup X_{\mathrm w}$. Throughout this paper we write $\PP$
for the measure defining the random variable $(X_{\mathrm b},X_{\mathrm w})$ in
the space~$\Omega^2$. The definition of the model strongly depends on the value
of $p$. Nevertheless, in all the proofs, the value of $p$ will be fixed, and we
do not mention the dependence on the underlying $p$ in our notation.
  
  In Voronoi percolation, we consider the Voronoi tiling $(V_z)_{z\in
    X}$ associated to $X$, and we are interested in the random
  coloring of the plane obtained by coloring \emph{black} the points
  in the cells corresponding to the black points of $X$, and
  \emph{white} the points in the cells corresponding to white points
  of $X$. In other words, the set of black points is the union of
  the cells $V_z$, $z\in X_{\mathrm b}$, and the set of white points is the
  union of the cells $V_z$, $z\in X_{\mathrm w}$. The points at the boundary
  between two cells of different colors are both black and white.

\paragraph{Crossing events.}
In our study, events will be simpler to define in terms of the colors
of the points in $\R^2$. For $S\subset \R^2$, we say that an event is
$S$-measurable if it is defined in terms of the colors in $S$.
Formally speaking, an event is $S$-measurable if it lies in the
sigma-algebra generated by the events $\{\text{All the points in $U$
  are black}\}$, $U\subset S$.

Let $A,B$ and $S$ be three subsets of $\R^2$ such that $A,B\subset S$.
We call \emph{black path from $A$ to $B$ in $S$} an injective
continuous map $\gamma: [0,1] \to S$ such that $\gamma(0)\in A$,
$\gamma(1)\in B$, and all the points in the Jordan arc $\gamma([0,1])$
are black. One can verify that the existence of a path from $A$ to $B$
in $S$ is an $S$-measurable event. In the same way, we define a
\emph{black circuit} in the annulus $A_{s,t}$, $s<t$ as a Jordan curve
included in $A_{s,t}$ such that the origin~$0$ is in its interior, and
all its point are black. \emph{White paths} and \emph{white circuits}
are defined analogously. Then, we define the circuit event by
\[
\cal A_s =\{\textrm{there exists a black circuit in the annulus
  $A_{s,2s}$} \}.\] Finally, for $\rho>0$ and $s>0$, we introduce
the crossing probability
\[
f_s(\rho)=\mathbf P \left[ \:
\begin{minipage}[c]{.54\linewidth}
  \textrm{there exists a black path from $\{0\}\times[0,s]$ to $\{\rho
    s\}\times[0,s]$ in the rectangle $[0,\rho s]\times[0,s]$}
\end{minipage}\: \right].
\]

\subsection{External ingredients}
\label{crossing:sec:external-ingredients}

\paragraph{Independence properties.}
One main difficulty in Voronoi percolation is the spatial dependency
between the colors of the points: given two fixed points in the plane,
there is a positive probability for them to lie on the same tile, thus
(for $0<p<1$) the probability that they are both black is larger than
$p^2$. Due to these correlations, we cannot use the standard ``lowest path''
argument discussed in the introduction. Nevertheless, the spatial
dependencies are only local and  the
color of a given point is determined with high probability by the
process restricted to a neighbourhood of it. More precisely, Lemma~3.2.\@ in~\cite{bollobas2006critical} states that the color of the points
in the box $B_s$ are determined with high probability by the process
$(X_{\mathrm b},X_{\mathrm w})$ restricted to $B_{s+2\sqrt{\log s}}$. In our approach,
this property is stronger than what we really need, and the following
lemma is sufficient. We consider the event
\[
 \cal F_s=\left \{\textrm{for every $z \in
A_{2s,4s}$, there exists some point $x \in X$ at distance $\mathrm
d(z,x) < s$}\right\}.
\]

\begin{lemma}
  \label{crossing:lem:decoralate}
  We have $\lim_{s\to\infty}\P{\cal F_s}=1$ and, for any
  $A_{2s,4s}$-measurable event $\cal E$, the event $\cal E \cap \cal
  F_s$ is measurable with respect to the restriction of $(X_{\mathrm
    b},X_{\mathrm w})$ to $A_{s,5s}$.
\end{lemma}

\begin{proof}
  Let us consider an absolute constant $C>0$ such that, for every
  $s\ge 1$, there exists a covering of $A_{2s,4s}$ by $C$ Euclidean
  balls of diameter $s$. Fix $s\ge 1$ and a covering of $A_{2s,4s}$
  by $C$ Euclidean balls of diameter $s$. Consider the event that each
  of these balls contains at least one point of the Poisson process
  $X$. Using that it is a sub-event of $\cal F_s$, we obtain
  \begin{equation}
    \P{\cal F_s}\ge 1-C \mathrm e^{-\pi s^2/4}.
  \end{equation}
  For the second part of the lemma, observe that the color of a point
  in $A_{2s,4s}$ is determined by the color of its closest point of
  the process $X$. When $\cal F_s$ holds, this point lies in
  $A_{s,5s}$. Thus, for any $U\subset A_{2s,4s}$, the event $\cal E_U$
  is measurable with respect to $(X_{\mathrm b}\cap A_{s,5s}, X_{\mathrm w} \cap
  A_{s,5s})$.
\end{proof}

\paragraph{FKG inequality.}
\label{crossing:sec:FKG}
The FKG inequality is an important tool allowing to ``glue'' black
paths. Its proof can be found in \cite{bollobas2010percolation}.
Before stating it, we need to define increasing events in the context
of Voronoi percolation. An event $\cal E$ is \emph{black-increasing}
if for any configurations $\omega=(\omega_b,\omega_w)$ and
$\omega'=(\omega_b',\omega_w')$, we have
\begin{equation}
  \left.
    \begin{array}[c]{c}
      \omega \in \cal E\\
      \omega_b\subset \omega_b' \text{ and }
      \omega_w \supset \omega_w'
    \end{array}
  \right \} \Rightarrow \omega'\in \cal E.
\end{equation}

\begin{proposition}[FKG inequality]
  \label{crossing:thm:fkg-inequality}
  Let $\cal E$ and  $\cal F$ be two black-increasing events, then
  \begin{equation}
    \P{\cal E\cap\cal F}\ge \P{\cal E}\P{\cal F}.    
  \end{equation}
\end{proposition}
\noindent The following  standard inequalities can be easily derived from
Proposition~\ref{crossing:thm:fkg-inequality}.

\begin{corollary}\label{crossing:cor:standardInequalities}
  Let $s\ge 1$.
  \begin{enumerate}
 \item\label{item:1} $f_s(2)\ge \P{\cal A_s}$,
  \item\label{item:2} $f_s(1+i\kappa)\ge f_s(1+\kappa)^i f_s(1)^{i-1}$ for any
    $\kappa>0$ and any $i\ge 1$,
  \item\label{item:3} $\P{\cal A_s}\ge f_s(4)^4$.
  \end{enumerate}
\end{corollary}

\subsection{Organization of the proof of Theorem~\ref{crossing:thm:main}}
We fix $0<p<1$, and assume that there exists a constant $c_0>0$ such that
for all $s\ge 1$,
\begin{equation}
  \label{crossing:eq:5}
  f_s(1)\ge c_0.
\end{equation}
Our goal is to prove that $\inf_{s\ge1}\P{\cal A_s}>0$, and then apply
Corollary~\ref{crossing:cor:standardInequalities}, Item~\ref{item:1} and \ref{item:2}. Rather than
studying only the sequence $(\P{\cal A_s})_{s\ge 1}$, we introduce at
each scale $s$ a real value $\alpha_{s}$ and study the pair $(\P{\cal
  A_s}, \alpha_{s})_{s\ge1}$ altogether. (The quantity $\alpha_s$ is
defined at the beginning of Section~\ref{crossing:sec:glueing}.)
\begin{description}
\item[Step 1: definition of good scales.] In Section~\ref{crossing:sec:glueing}, a geometric construction valid
  only when $\alpha_{s}\le2\alpha_{2s/3}$ provides a RSW-result at
  scale $s$. We will refer to such scale as a ``good scale''.
\item[Step 2: renormalization.] In Section~\ref{crossing:sec:all-scales-are}, we use the independence
  properties of the model to show that the good scales are close to
  each other. More precisely, we construct an infinite sequence
  $s_1, s_2, \ldots$ of good scales such that $4s_i\le s_{i+1}\le C
  s_i$.
\end{description}

Throughout the proof, we will work with constants. By convention, they
are elements of $(0,\infty)$, and they do not depend on any parameter
of the model. In particular, they never depend on the scale parameter
$s$. These constants will generally be denoted by $c_0,c_1,\ldots$
or $C_0,C_1,\ldots$ (depending on whether they have to be thought
small or large).
\section{Gluing at good scales}
\label{crossing:sec:glueing}


Fix $s\ge 1$. For $-s/2 \leq \alpha \leq \beta \leq s/2$, define $\cal
H_s(\alpha,\beta)$ to be the event that there exists a black path in
the square $B_{s/2}$, from the left side to
$\{s/2\}\times[\alpha,\beta]$ (see Fig.~\ref{crossing:fig:eventH} for
an illustration). For $0\le \alpha \leq s/2$, define $\cal
X_s(\alpha)$ to be the event that there exist
\begin{itemize}
\item a black path $\gamma_{-1}$  in $B_{s/2}$ from $\{-s/2\}\times [-s/2,-\alpha]$ to $\{-s/2\}\times
  [\alpha,s/2]$,
\item a black path $\gamma_1$  in $B_{s/2}$ from $\{s/2\}\times [-s/2,-\alpha]$ to $\{s/2\}\times
  [\alpha,s/2]$,
\item a black path  in $B_{s/2}$ from $\gamma_{-1}$ to $\gamma_1$.
\end{itemize}
\begin{figure}[htbp]
\centering
\hfill
\begin{minipage}[t]{.4\linewidth}
\centering
  \includegraphics[width=\linewidth]{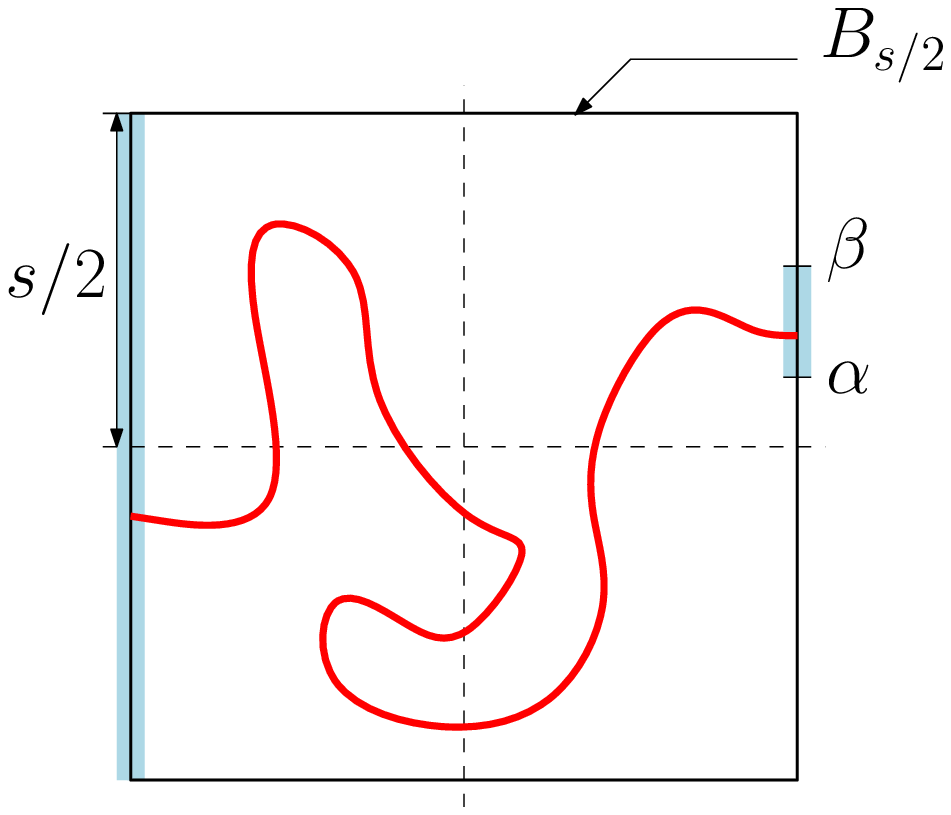}
  \caption{The event $\cal H_s(\alpha,\beta)$}
  \label{crossing:fig:eventH}
\end{minipage}
\hfill
\begin{minipage}[t]{.4\linewidth}
\centering
  \includegraphics[width=\linewidth]{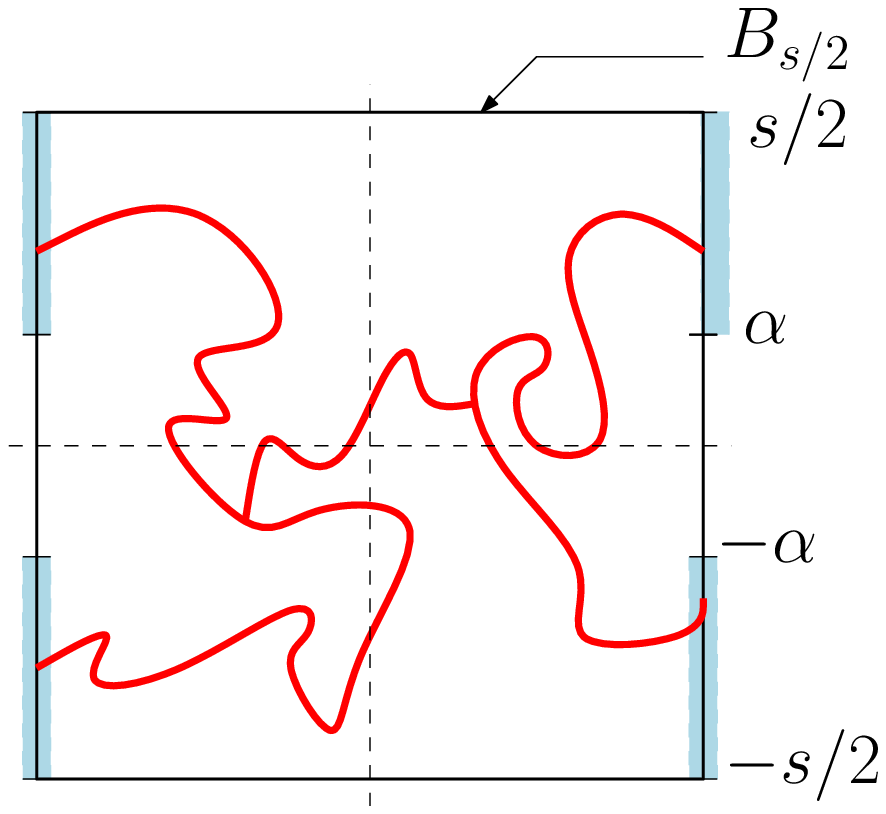}
  \caption{The event $\cal X_s(\alpha)$}
  \label{crossing:fig:eventX}
\end{minipage}
\hfill-
\end{figure}
Let $\phi_s:[0,s/2]\to [-1,1]$ be the function defined by
  \begin{equation}
    \phi_s(\alpha)=\P{\cal H_s(0,\alpha)} -\P{\cal
      H_s(\alpha,s/2)},\quad 0\le \alpha\le s/2.
  \end{equation}
  One can verify that $\phi_s$ is continuous, strictly increasing, and
  satisfies $\phi_s(0)\le 0$. In addition, if we assume that
  Equation~\eqref{crossing:eq:5} holds, then symmetry implies
  $\phi_s(s/2)\ge c_0/2$.

  \begin{lemma}
    \label{crossing:lem:split}
    Assume that Equation~\eqref{crossing:eq:5} holds. Then for every $s\ge
    1 $, there exists \mbox{$\alpha_{s} \in [0,s/4]$} such that the following two
    properties hold.
    \begin{enumerate}[\bf(P1)]
    \item\label{crossing:item:1} For all $0\le \alpha\le \alpha_{s}$, $\P{\cal X_s(\alpha)}\ge
    c_1$.
  \item\label{crossing:item:2} If $\alpha_{s}<s/4$, then for all
    $\alpha_{s}\le \alpha\le s/2$,
    $\P{\cal H_s(0,\alpha)}\ge c_0/4 +\P{\cal H_s(\alpha,s/2)}$.
  \end{enumerate}
\end{lemma}
\noindent In the rest of the paper, Equation~\eqref{crossing:eq:5} is
always assumed to hold, and we fix for every $s\ge 1$ a real number
$\alpha_s \in [0,s/4]$ satisfying \ref{crossing:item:1} and
\ref{crossing:item:2} above.

\begin{proof}
  The properties of $\phi_s$ allow us to define
  \begin{equation}
    \alpha_s=\min\big(\phi_s^{-1}(c_0/4),s/4\big).
  \end{equation}
  With this definition, Property~\textbf{(P2)} is
  clearly satisfied. We only need to show that Property~\textbf{(P1)} holds. If
  $\alpha \le \alpha_{s}$, our hypothesis~\eqref{crossing:eq:5} and
  symmetries imply that \vspace{-5pt}
  \begin{spreadlines}{.32\baselineskip}
    \begin{align}
      c_0&\le 2 \P{\cal H_s(0,s/2)}\\
      &\le 2\P{\cal H_s(0,\alpha)}+2\P{\cal H_s(\alpha,s/2)}\\
      &\le 2 \phi_s(\alpha) +4\P{\cal H_s(\alpha,s/2)}\\
      &\le c_0/2 + 4\P{\cal H_s(\alpha,s/2)}.
    \end{align}
  \end{spreadlines}
  We obtain, for every $\alpha\le \alpha_{s}$,
  \begin{equation}
    \P{\cal H(\alpha,s/2)}\ge c_0/8.
  \end{equation}
  A sub-event of $\cal X_s(\alpha)$ can be obtained by intersecting four
  symmetric versions of $\cal H_s(\alpha,s/2)$ with the event that there
  exists a top-down crossing in $B_{s/2}$. The FKG inequality implies then
\begin{equation}
  \P{\cal X_s(\alpha)}\geq c_0(c_0/8)^4.
\end{equation}
This concludes the first part of the lemma with $c_1=  c_0(c_0/8)^4$.
\end{proof}

\begin{lemma}
  \label{crossing:lem:glue} 
  There exists $c_2>0$ such that for all $s\ge 2$, the inequality
  $\alpha_{s}\le 2\,\alpha_{2s/3}$ implies
  \begin{equation}
    \label{crossing:eq:22}
   \P{\cal A_{s}}\ge c_2.
  \end{equation}
\end{lemma}
\begin{proof}
  We first treat the case $\alpha_{s}=s/4$. (In this case, we directly
  prove that \eqref{crossing:eq:22} holds, without using the
  hypothesis $\alpha_{s}\le 2\,\alpha_{2s/3}$.) By Property
  \ref{crossing:item:1} of Lemma~\ref{crossing:lem:split}, we have
  $\P{\cal X_s(s/4)}\ge c_1$, and it is easy to create a black
  crossing in a long rectangle. Consider for $i=0,\ldots,4$ the event
  $\cal E_i$ that there exists a black path from
  $\{0\}\times[(i-1)s/2,is/2]$ to $\{0\}\times[(i+1)s/2,(i+2)s/2]$ in
  the strip $[0,s]\times\R$. For every $i$, the event $\cal E_i$ has
  probability larger than $\P{\cal X_s(s/4)}$, and when all of them
  occur, it implies a vertical black crossing in the rectangle
  $[0,s]\times[0,2s]$. FKG inequality implies that $f_{s}(2)\ge
  c_1^5$. And hence, by Items~\ref{item:2} and \ref{item:3} of
  Corollary~\ref{crossing:cor:standardInequalities},
  \begin{equation}
       \P{\cal A_s}\ge (c_1^{15}c_0^2)^4.
  \end{equation}
  Now, let $s$ be such that $\alpha_{s}\leq 2 \alpha_{2s/3}$ and
  $\alpha_{s}< s/4$. We use the event $\cal X_{2s/3}(\alpha_{2s/3})$
  to connect at scale $2s/3$ two crossings at scale $s$. Consider the
  two squares $R=(-s/6,-\alpha_{2s/3})+B_{s/2}$ and
  $R'=(s/6,-\alpha_{2s/3})+B_{s/2}$. Notice that $B_{s/3} \subset R$
  and $B_{s/3}\subset R'$ since $\alpha_{2s/3}\leq s/6$. Let $\cal E$
  be the event that there exists a black path from left to
  $\{s/3\}\times[-\alpha_{2s/3},\alpha_{2s/3}]$ in $R$. Similarly,
  define $\cal E'$ as the event that there exists a black path from
  $\{-s/3\}\times[-\alpha_{2s/3},\alpha_{2s/3}]$ to right in $R'$.
  Since $\alpha_{s}\le 2\alpha_{2s/3}\le s/2$ and $\alpha_{s}< s/4$,
  Property~\ref{crossing:item:2} in Lemma~\ref{crossing:lem:split}
  ensures that both events $\cal E$ and $\cal E'$ have probabilities
  larger than $c_0/4$. Recall that, by Property~\ref{crossing:item:1}
  in Lemma~\ref{crossing:lem:split}, the event $\cal
  X_{2s/3}(\alpha_{2s/3})$ has probability larger than $c_1$.
         
  When the three events $\cal X_{2s/3}(\alpha_{2s/3})$, $\cal E$ and
  $\cal E'$ occur, a black path must exist from left to right in the
  rectangle $R\cup R'$ (see Fig.~\ref{crossing:fig:constr3}). The
  rectangle $R\cup R'$ has aspect ratio $4/3$, and FKG inequality
  implies
  \begin{align}
    f_{s}(4/3)&\geq\P{\cal X_{2s/3}(\alpha_{2s/3}) \cap \cal E\cap\cal E'}\\
    &\ge c_1\left(\frac{c_0}4\right)^2.
  \end{align}
  Then, as above, we use Items~\ref{item:2} and \ref{item:3} of
  Corollary~\ref{crossing:cor:standardInequalities} to conclude the
  proof.

  \begin{figure}[htp]
    \centering
    \includegraphics[width=.58\linewidth]{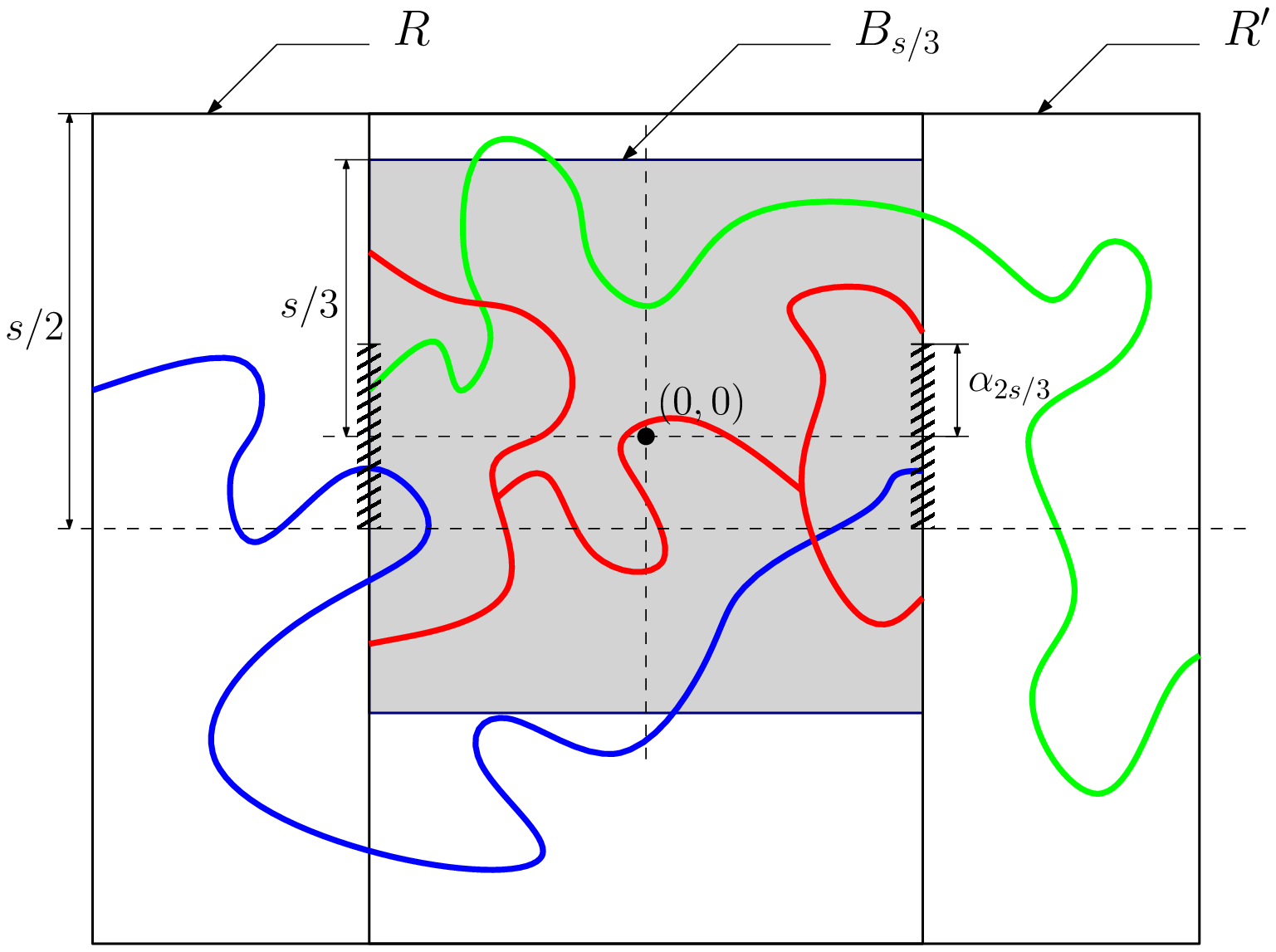}
    \caption{The simultaneous occurrence of $\cal X_{2s/3}(\alpha_{2s/3})$, $\cal E$ and $\cal E'$
      implies the existence of a horizontal crossing in $R\cup R'$.}
    \label{crossing:fig:constr3}
  \end{figure}
\end{proof}

\noindent\textbf{Comment.} Lemma~\ref{crossing:lem:glue} is central in our
approach. As soon as the inequality
\begin{equation}
\alpha_{s}\le 2\alpha_{2s/3}\label{eq:12},
\end{equation}
holds, we obtain a ``RSW statement'' at scale $s$. So far, we have used only
positive association, and the invariance of the measure under symmetries. The
independence property of Lemma~\ref{crossing:lem:decoralate} will be useful in
the next section, to show that the inequality of Equation~\eqref{eq:12} holds at
every scale, roughly. Before, let us notice that we already have that the
inequality of Equation~\eqref{eq:12} holds for infinitely many scales. Indeed,
$\alpha_s$ is always smaller than $s$, thus it cannot grow super-linearly.
Hence, Lemma~\ref{crossing:lem:glue} implies
\begin{equation}
  \label{eq:13}
  \limsup_{s\to\infty}\P{\cal A_s}\ge c_2.
\end{equation}
In other words, we already obtain the weak RSW of Bollobás and
Riordan. Notice that to prove this result we used only the positive
association, and the invariance of the measure under symmetries.

\section{Proof of Theorem~\ref{crossing:thm:main}}
\label{crossing:sec:all-scales-are}

\begin{lemma}
  \label{crossing:lem:RSWafterGood}
  There exists $c_3>0$ such that the following holds for every  $s\ge 1$ and
  $t\ge 4 s$.
  \begin{equation}
 \text{If } \P{\cal A_s}\ge c_2\text{ and }\alpha_{t}\le s,\text{ then } \P{\cal A_t}\ge c_3.    
  \end{equation}
\end{lemma}

\begin{proof}
 Let $s\ge 1$ and
  $t\ge 4s$. Assume that $\P{\cal A_s}\ge c_2$ and $\alpha_{t}\le s$.
  Consider the event that there exist
  \begin{itemize}
  \item a black path from left to $\{0\}\times[0,s]$ in the square
    $[-t,0]\times[-t/2,t/2]$,
  \item a black path from $\{0\}\times[0,s]$ to right in the square
    $[0,t]\times[-t/2,t/2]$,
  \item and a black circuit in the annulus $A_{s,2s}$. 
  \end{itemize}
  Since $\alpha_{t}\le s$, Lemma~\ref{crossing:lem:split} implies that each of
  the first two paths exists with probability larger than $c_0/4$. When the
  event depicted above occurs, it implies the existence of an
  horizontal black crossing in the rectangle $[-t,t]\times[-t/2,t/2]$.
  Using the FKG inequality, we obtain
  \begin{equation}
    f_{t}(2)\ge \left(\frac{c_0}4\right)^{2}c_2.
  \end{equation}
  The standard inequalities of
  Corollary~\ref{crossing:cor:standardInequalities} allow to conclude that 
  \begin{align}
    \P{\cal A_t}\ge c_3,
  \end{align}
for some constant $c_3>0$.
\end{proof}


The next lemma is the part of the proof that uses the independence
properties of the model. Before stating it, we invoke Lemma
\ref{crossing:lem:decoralate} and define $s_0$ such that
\begin{equation}
  \P{\cal
    F_s}\ge 1-c_3/2\quad\text{for all $s\ge s_0$}.\label{eq:3}
\end{equation}

\begin{lemma}
  \label{crossing:lem:circuitArg}
  Define a constant $C_1\ge 4$ large enough, so that
  \begin{equation}\label{eq:4}
    (1-c_3/2)^{\lfloor\log_5(C_1)\rfloor}<c_0/4    
 \end{equation}
 Let $s\ge s_0$ such that $\P{\cal A_s}\ge c_2$, then there exists
 $s'\in [4s, C_1 s]$ such that $\alpha_{s'}\ge s$.
\end{lemma}

\begin{proof}
  Let $s\ge s_0$ such that $ \P{\cal A_s}\ge c_2$. Assume for
  contradiction that $\alpha_{t}< s$ for all $4s\le t\le C_1s$. For
  $t=C_1 s$, this implies that $\alpha_{C_1 s}<C_1s/4$ and
  $\alpha_{C_1s}\le s\le C_1s/2$. Hence, by Property~\ref{crossing:item:2} in
  Lemma~\ref{crossing:lem:split}, we have
  \begin{equation}
    \label{eq:5}
     \P{\cal H_{C_1 s}(0,s)} -\P{\cal H_{C_1 s}(s,C_1s)}\ge c_0/4.
  \end{equation}
  
  Let $1\le i\le \lfloor \log_5(C_1) \rfloor$.  Since $\P{\cal A_s}\ge c_0$ and
  $\alpha_{5^is}\le s$, Lemma~\ref{crossing:lem:RSWafterGood} applied
  with $t=5^is$ implies that $\P{\cal A_{5^is}}\ge c_3$. Together with
  Equation~\eqref{eq:3}, we find
  \begin{align}
    \label{eq:6}
    \P{\cal A_{5^is} \cap \cal F_{5^is}}\ge c_3/2.
  \end{align}   
  Let $\cal E$ be the event that there exists a black circuit in the
  annulus $A_{s,C_1 s}$. This happens as soon as $\cal A_{5^is} \cap \cal F_{5^is}$
  occurs for some $1\le i\le \lfloor\log_5(C_1)\rfloor$. By
  Lemma~\ref{crossing:lem:decoralate}, these events
  are independent, and we find
  \begin{align}
    \P{\cal E^c} &\le \mathbf P \big [ \bigcap_{1\le i \le \lfloor\log_5(C_1)\rfloor} \!\!\!\!\!\!\!(\cal
        A_{5^is}\cap \cal F_{5^is})^c \ \big] \\
      &\le (1-c_3/2)^{\lfloor\log_5(C_1)\rfloor}\\
      &< c_0/4. \label{eq:7}
  \end{align}
  In the second inequality, we applied the independence property of
  Lemma~\ref{crossing:lem:decoralate} together with Equation~\eqref{eq:6}, and
  in the third inequality we used Equation~\eqref{eq:4}.

\noindent  Now, consider the event that in the square $[-C_1s,0]\times[-C_1s/2,C_1s/2]$,
  \begin{itemize}
  \item there exists a black path from left to $\{0\}\times[0,s]$, but
  \item there is NO black path from left to $\{0\}\times[s,C_1s]$.
  \end{itemize}
  By translation invariance, this has probability larger than $\P{\cal H_{C_1
      s}(0,s)} -\P{\cal H_{C_1 s}(s,C_1s)}$. And when this holds, there cannot
  exist a black circuit in the annulus $A_{s,C_1 s}$. Using
  Equation~\eqref{eq:7}, we obtain
  \begin{equation}
    \label{eq:8}
    \P{\cal
    H_{C_1 s}(0,s)} -\P{\cal H_{C_1 s}(s,C_1s)} <c_0/4,
  \end{equation}
    which contradicts Equation~\eqref{eq:5}.
  \end{proof}

\begin{lemma}
  \label{crossing:lem:sequenceSi}
  There exist a constant $C_3\geq 4$ and an infinite sequence
  $s_1,\,s_2\ldots$ of scales such that for all $i\ge 1$, 
  \begin{itemize}
  \item $4s_i \le s_{i+1}\le C_3 s_i$,
  \item $\P{\cal A_{s_i}}\ge c_2$.
  \end{itemize}
\end{lemma}
\begin{proof}
  Since $\alpha_{s}\le s$, the sequence $\alpha_{s}$ cannot grow
  super-linearly, and there must exist $s_1\ge s_0$ such that
  $\alpha_{s_1}\le 2\alpha_{2s_1/3}$. By
  Lemma~\ref{crossing:lem:glue}, we obtain that $\P{\cal
    A_{s_1}}\ge c_2$. Therefore, Lemma~\ref{crossing:lem:circuitArg}
  implies the existence of $s_1'\in [4s_1,C_1 s_1]$ such that
  \begin{equation}
    \label{crossing:eq:3}
    \alpha_{s_1'}\ge s_1'/C_1.
  \end{equation}
  Then, there must exist $s_2\in[s_1', C_1^{\log_{4/3}(3/2)}s_1']$
  such that $\alpha_{s_2}\le 2\alpha_{2s_2/3}$, otherwise the bound
  $\alpha_{s}\le s$ would be contradicted. Define
  $C_3=C_1^{1+\log_{4/3}(3/2)}$. We have $s_2\in[4s_1,C_3s_1]$ and we
  find from Lemma~\ref{crossing:lem:glue} that $\P{\cal   A_{s_2}}\ge
  c_2$.
 
  The constant $C_3$ is independent of the scale, we can
  thus iterate the construction above, and find by induction
  $s_3,s_4,\ldots$ 
\end{proof}
Theorem~\ref{crossing:sec:voronoi-percolation} follows easily from
Lemma~\ref{crossing:lem:sequenceSi} and the standard inequalities of
Corollary~\ref{crossing:cor:standardInequalities}.


\section{Proofs of Theorems~\ref{crossing:thm:highProb} and \ref{thm:2}}
\label{sec:theor-impl-theor}

To prove Theorem~\ref{crossing:thm:highProb}, we will need the
following proposition, called the ``square root trick''. It is a
standard consequence of the FKG inequality (see e.g.
\cite{grimmett1999}).

\begin{proposition}{\bf (square root trick)}
   Let $\cal E_1,\ldots,\cal E_k$ be increasing events, and write $\cal
   E=\cal E_1\cup\cdots\cup\cal E_k$. Then, the
  following inequality holds:
  \begin{equation}
    \max_{1\le i\le k}\Pp{\cal E_i} \geq 1-\left(1-\Pp{\cal E }\right)^{1/k}.
  \end{equation}
\end{proposition}

\begin{proof}[Proof of Theorem~\ref{crossing:thm:highProb}]
  We assume that $f_s(1)$ converges to $1$ when $s$ tends to infinity.
  We prove that $f_{s}(4/3)$ converges also to $1$. The more general
  statement of Theorem~\ref{crossing:thm:highProb} can be then
  obtained by using the standard inequalities of
  Corollary~\ref{crossing:cor:standardInequalities}.
  
  Fix $\eps>0$. By Theorem~\ref{crossing:thm:main}, there exists a
  constant $c>0$ such that $\P{\cal A_s}>c$ for all $s\ge 1$. With the
  same argument we used to obtain~\eqref{eq:7} in the proof of
  Lemma~\ref{crossing:lem:circuitArg}, we can use
  Lemma~\ref{crossing:lem:decoralate} to show the following. There
  exists $\eta>0$, such that for every $s$ large enough,
  \begin{equation}
    \label{eq:1}
    \P{\text{There exists a black circuit surrounding $B_{\eta s}$ in
        $A_{\eta s, s/4}$}}>1-\eps.   
  \end{equation}
  We can cover the right side of $B_{s/2}$ with less than
  $\lfloor1/\eta\rfloor$ segments of length $2\eta s$. By
  the square root trick, there exists $y_s\in [-s/2,s/2]$ such that 
  \begin{equation}
    \label{eq:2}
    \P{\cal H_s(y_s-\eta s,y_s+\eta s)}\ge 1-(1-f_s(1))^{1/\eta}.
  \end{equation}
  Consider the event that there exist
  \begin{itemize}
  \item a black path from left to $\{s/2\}\times[ y_s-\eta s,y_s+\eta
    s]$ in $B_s$,
    \item a black path from $\{s/2\}\times[ y_s-\eta s,y_s+\eta
    s]$ to right in $(s,0)+B_{s/2}$,
  \item a black circuit in the annulus
        $(s/2,y_s)+A_{\eta s, s/4}$. 
  \end{itemize}
  When this event occurs, it implies the existence of a left-right
  crossing in the rectangle $[-s/2,3s/2]\times[-3s/4,3s/4]$. By the FKG
  inequality, we obtain that for all $s$ large enough
  \begin{equation}
    f_{3s/2}(4/3)\ge \left(1-(1-f_s(1))^{1/\eta}\right)^{2}(1-\eps).
  \end{equation} 
  This implies that $\liminf_{s\to\infty}f_{s}(4/3)\ge1-\eps$. 
\end{proof}

\begin{proof}[Proof of Theorem \ref{thm:2}] In this proof, we set $p=1/2$. The
  derivation of the box-crossing property from the RSW result of
  Theorem~\ref{crossing:thm:main} is standard. We only sketch the proof, and
  refer to \cite{bollobas2010percolation}, Chapter 8, for more details. As
  mentioned in the introduction, self-duality and symmetries of the model imply
  that $f_s(1)=1/2$, for every $s\ge 1$. Theorem~\ref{crossing:thm:main} implies
  directly that for every $\rho\ge 1$ there exists $c(\rho)$ such that
  \begin{equation}
    \label{eq:11}
    c(\rho)\le f_s(\rho) \le 1-c(\rho).
  \end{equation}
  The upper bound follows from the trivial inequality $f_s(\rho)\le
  1/2$ when $\rho \ge1$.

  The proof of Equation~\eqref{eq:11} for $\rho<1$ can be then derived
  from the case $\rho>1$. Using the symmetric roles played by black
  and white, and the fact that in a rectangle $R$, exactly one of
  these cases occurs. Either there is a black horizontal crossing in
  $R$, or there exists a white horizontal crossing in $R$.

  The proof of the polynomial decay for the one-arm exponent follows
  from a standard ``circuit argument'' that we already used in the
  proofs of Lemma~\ref{crossing:lem:circuitArg} and Theorem
  \ref{crossing:thm:highProb}. By the box-crossing property and
  Lemma~\ref{crossing:lem:decoralate}, there exists a constant $c>0$
  such that the event $\cal A_s\cap \cal F_s$ has probability larger
  than $c$, for all $s\ge1$. Let $1\le s\le t$. Considering the
  independent events $\cal A_{5^is} \cap \cal F_{5^is}$, $0\le i \le \lfloor
  \log_{5}(t/s)\rfloor$, we find that there exists a black circuit in
  $A_{s,t}$ with probability larger than \[1-(1-c)^{\log_{5}(t/s)}.\]
  Therefore, by duality, there exists a white path from $B_s$ to the
  boundary of $B_t$ with probability smaller than
  $(1-c)^{\log_{5}(t/s)}$. This concludes the proof, since a white
  path from $B_s$ to the boundary of $B_t$ exists with probability
  exactly $\pi_1(s,t)$.
\end{proof}

\section*{Acknowledgments}
The author is grateful to Vincent Beffara and Daniel Ahlberg for
useful discussions, and to Matan Harel, Hugo Duminil-Copin, and Jacob
van den Berg for their comments on the earlier versions of this paper.
I also thank the anonymous referee for useful comments. This work
was achieved during the PhD of the author at ENS Lyon. It was
supported in part by the ANR grant MAC2 (ANR-10-BLAN-0123), and the
Swiss NSF.

\nocite{grimmett1999}
\nocite{martineau2013locality}

\bibliographystyle{alphaNames}
\bibliography{references}

\begin{thebibliography}{vdBBV08}

\bibitem[Aiz98]{aizenman1998scaling}
M.~Aizenman.
\newblock Scaling limit for the incipient spanning clusters.
\newblock In {\em {\mbox{Mathematics}} of {M}ultiscale {M}aterials}, pages
  1--24. Springer, 1998.

\bibitem[Ale96]{alexander1996rsw}
K.~S Alexander.
\newblock The {RSW} theorem for continuum percolation and the {CLT} for
  {E}uclidean minimal spanning trees.
\newblock {\em The Annals of Applied Probability}, 6(2):466--494, 1996.

\bibitem[BBQ05]{balister2005percolation}
P.~Balister, B.~Bollob{\'a}s, and A.~Quas.
\newblock Percolation in {V}oronoi tilings.
\newblock {\em {R}andom {S}tructures {A}lgorithms}, 26(3):310--318, 2005.

\bibitem[BD12]{beffara2012self}
V.~Beffara and H.~{Duminil-Copin}.
\newblock The self-dual point of the two-dimensional random-cluster model is
  critical for $q\ge1$.
\newblock {\em Probability Theory and Related Fields}, 153(3-4):511--542, 2012.

\bibitem[BS98]{benjamini1998conformal}
I.~Benjamini and O.~Schramm.
\newblock Conformal invariance of {V}oronoi {\mbox{percolation}}.
\newblock {\em Communications in {M}athematical {P}hysics}, 197(1):75--107,
  1998.

\bibitem[BR06]{bollobas2006critical}
B.~Bollob{\'a}s and O.~Riordan.
\newblock The critical probability for random {V}oronoi \mbox{percolation} in
  the plane is 1/2.
\newblock {\em Probability {T}heory and {R}elated {F}ields}, 136(3):417--468,
  2006.

\bibitem[BR10]{bollobas2010percolation}
B.~Bollobás and O.~Riordan.
\newblock Percolation on self-dual polygon {\mbox{configurations}}.
\newblock In Imre Bárány, József Solymosi, and Gábor Sági, editors, {\em
  An Irregular Mind}, volume~21 of {\em Bolyai Society Mathematical Studies},
  pages 131--217. Springer Berlin Heidelberg, 2010.

\bibitem[DNS15]{DNS12}
M.~Damron, C.~M. Newman, and V.~Sidoravicius.
\newblock Absence of site percolation at criticality in $\mathbb{Z}^2
  \times\{0,1\}$.
\newblock {\em Random Structures and \mbox{Algorithms}}, 47(2):328--340, 2015.

\bibitem[DHN11]{duminil2011connection}
H.~{Duminil-Copin}, C.~Hongler, and P.~Nolin.
\newblock Connection probabilities and {RSW}-type bounds for the
  two-dimensional {FK} {I}sing model.
\newblock {\em Communications on Pure and Applied Mathematics},
  64(9):1165--1198, 2011.

\bibitem[DST14]{duminil2013absence}
H.~{Duminil-Copin}, V.~Sidoravicius, and V.~Tassion.
\newblock Absence of {\mbox{percolation}} for critical {B}ernoulli percolation
  on slabs.
\newblock {\em Preprint arXiv:1401.7130, To appear in Comm. Pure Appl. Math.},
  2014.

\bibitem[DST15]{duminil2014continuity}
H.~{Duminil-Copin}, V.~{Sidoravicius}, and V.~{Tassion}.
\newblock {Continuity of the phase transition for planar random-cluster and
  Potts models with $1\le q\le4$}.
\newblock {\em ArXiv e-prints 1505.04159}, May 2015.

\bibitem[Gri99]{grimmett1999}
G.~R. Grimmett.
\newblock {\em Percolation}, volume 321.
\newblock Springer Verlag, 1999.

\bibitem[GM13]{grimmett2013}
G.~R. Grimmett and I.~Manolescu.
\newblock Universality for bond \mbox{percolation} in two dimensions.
\newblock {\em The {A}nnals of {P}robability}, 41(5):3261--3283, 09 2013.

\bibitem[Kes82]{kesten1982percolation}
H.~Kesten.
\newblock {\em Percolation theory for mathematicians}, volume~2 of {\em
  Progress in Probability and Statistics}.
\newblock Birkh\"auser Boston, Mass., 1982.

\bibitem[MT13]{martineau2013locality}
S.~Martineau and V.~Tassion.
\newblock Locality of percolation for abelian {C}ayley graphs.
\newblock {\em Preprint arXiv:1312.1946}, 2013.

\bibitem[Roy90]{roy1990russo}
R.~Roy.
\newblock The {R}usso-{S}eymour-{W}elsh theorem and the equality of critical
  \mbox{densities} and the "dual" critical \mbox{densities} for continuum
  percolation on $\mathbb{R}^2$.
\newblock {\em The Annals of Probability}, pages 1563--1575, 1990.

\bibitem[Rus78]{russo1978note}
L.~Russo.
\newblock A note on percolation.
\newblock {\em Zeitschrift f{\"u}r Wahrscheinlichkeitstheorie und verwandte
  Gebiete}, 43(1):39--48, 1978.

\bibitem[Rus81]{russo1981critical}
L.~Russo.
\newblock On the critical percolation probabilities.
\newblock {\em Zeitschrift f{\"u}r \mbox{Wahrscheinlichkeitstheorie} und
  verwandte Gebiete}, 56(2):229--237, 1981.

\bibitem[SW78]{seymour1978percolation}
P.~D. Seymour and D.~J.~A. Welsh.
\newblock Percolation probabilities on the square lattice.
\newblock {\em Annals of {D}iscrete {M}athematics}, 3:227--245, 1978.
\newblock Advances in graph theory (Cambridge Combinatorial Conf., Trinity
  College, Cambridge, 1977).

\bibitem[VW93]{vahidi1993upper}
M.~Q. {Vahidi-Asl} and J.~C. Wierman.
\newblock Upper and lower bounds for the route length of first-passage
  percolation in {V}oronoi tessellations.
\newblock {\em Bulletin of the Iranian Mathematical Society}, 19(1):15--28,
  1993.

\bibitem[vdBBV08]{van2008box}
J.~van~den Berg, R.~Brouwer, and B.~V{\'a}gv{\"o}lgyi.
\newblock Box-crossings and \mbox{continuity} results for self-destructive
  percolation in the plane.
\newblock In {\em In and out of \mbox{equilibrium} 2}, pages 117--135.
  Springer, 2008.

\end{thebibliography}
\vfill
\raggedleft
{\sc 
  Département de Mathématiques, Université de Genève
 
  Genève, Switzerland
 }
  
\texttt{Vincent.Tassion@unige.ch}
\end{document}